\documentclass[leqno,12pt]{amsart}
\usepackage{latexsym}
\usepackage{mathrsfs}
\usepackage{amsmath}
\usepackage{amssymb}
\usepackage[bookmarks]{hyperref}
\usepackage{pstricks,pst-plot}

\newtheorem{theorem}{Theorem}[section]
\newtheorem{lemma}[theorem]{Lemma}
\newtheorem{corollary}[theorem]{Corollary}
\newtheorem{proposition}[theorem]{Proposition}

\theoremstyle{definition}

\newtheorem{definition}[theorem]{Definition}

\newcommand{\C}{\mathbb{C}}

\newcommand{\R}{\mathbb{R}}

\newcommand{\bthm}{\begin{theorem}}
\newcommand{\ethm}{\end{theorem}}
\newcommand{\blem}{\begin{lemma}}
\newcommand{\elem}{\end{lemma}}
\newcommand{\bcor}{\begin{corollary}}
\newcommand{\ecor}{\end{corollary}}
\newcommand{\bprop}{\begin{proposition}}
\newcommand{\eprop}{\end{proposition}}
\newcommand{\bdefn}{\begin{definition}}
\newcommand{\edefn}{\end{definition}}
\newcommand{\bpf}{\begin{proof}}
\newcommand{\epf}{\end{proof}}

\def\vep {\varepsilon}
\def \sm {\setminus}

\newcommand{\ra}{\rightarrow}

\newcommand{\ol}{\overline}

\begin{document}
\title{Every connected bounded domain of holomorphy has connected boundary}

\author{Alexander J. Izzo}
\address{Department of Mathematics and Statistics, Bowling Green State University, Bowling Green, OH 43403}
\email{aizzo@bgsu.edu}
\thanks{The author was partially supported by NSF Grant DMS-1856010.}


\begin{abstract}
It is shown that every connected, bounded domain of holomorphy in $\C^n$, $n\geq 2$, has connected boundary.
\end{abstract}

\maketitle

%
%
%
%

\section{Introduction}

In this short note we prove the following theorem.  It is surprising that this result seems not to have appeared in the literature earlier.

\begin{theorem}\label{main-theorem}
Every connected, bounded domain of holomorphy in $\C^n$, $n\geq 2$, has connected boundary.
\end{theorem}

We will show that the theorem follows readily from the well-known Hartog's extension theorem and a simple lemma about connected sets in Euclidean space.  The form of Hartog's extension theorem that we will use is the following.  It can be found, with proof, in many sources, for instance \cite[Theorem~E6]{Gunning} and \cite[Theorem~16.3.6]{Rudin}.

\begin{theorem}[Hartog's extension theorem]
If $K$ is a compact subset of an open set $U\subset\C^n$ with $n\geq 2$ and if $U\sm K$ is connected, then any function that is holomorphic on $U\sm K$ extends to a function that is holomorphic on $U$.
\end{theorem}

The lemma about connected sets in Euclidean space that we need is the following.

\begin{lemma}
Let $E$ be a bounded set in $\R^N$, $N\geq 2$, that is connected and has connected complement.  Then the boundary of $E$ is connected.
\end{lemma}

\begin{proof}
Denote the complement of $E$ in $\R^N$ by $F$.  Given $\vep>0$ and a set $A$ in $\R^N$, denote by $A(\vep)$ the $\vep$-neighborhood of $A$, i.e., the set of points in $\R^N$ whose distance from $A$ is strictly less than $\vep$.  Then $E(\vep)$ and $F(\vep)$ are connected, open sets whose union is $\R^N$.  Consider the following segment of the the Mayer-Vietoris exact sequence in reduced singular homology (with integer coefficients) for the excisive couple $(E(\vep), F(\vep))$:
$$\cdots\ra H_1(\R^N)\ra \tilde H_0\bigl(E(\vep)\cap F(\vep)\bigr)\ra \tilde H_0\bigl(E(\vep)\bigr) \oplus \tilde H_0\bigl(F(\vep)\bigr)\ra\cdots.$$
The groups at the ends are zero, so the group in the middle must be zero also.  Thus $E(\vep) \cap F(\vep)$ is connected, and hence so is the closure $\ol{E(\vep)\cap F(\vep)}$ of $E(\vep) \cap F(\vep)$.

The collection $\bigl\{ \ol{E(\vep)\cap F(\vep)}: \vep>0\bigr\}$ is a totally ordered collection of compact, connected sets.  Therefore, the intersection $\bigcap\limits_{\vep>0} \ol{
E(\vep)\cap F(\vep) 
}$ is connected as well.  But this intersection is exactly the boundary of $E$.
\end{proof}

\begin{proof}[Proof of Theorem~\ref{main-theorem}]
Let $\Omega$ be a bounded, connected domain of holomorphy in $\C^n$, $n\geq 2$.  By the lemma, it suffices to show that the complement of $\Omega$ is connected.  Assume to get a contradiction that $\C^n\sm \Omega$ is not connected.  Let $K$ and $L$ form a separation of $\C^n\sm \Omega$ with the unbounded component of $\C^n\sm \Omega$ contained in $L$.  Set $U=\C^n\sm L$.  Then $K$ is a compact subset of the open set $U\subset\C^n$, and $U\sm K=\Omega$ is connected.  Thus Hartog's extension theorem shows that every holomorphic function on $\Omega$ extends to a holomorphic function on $U$,  contradicting the hypothesis that $\Omega$ is a domain of holomorphy.  Thus $\C^n\sm\Omega$ must be connected.
\end{proof}

\end{document}